\documentclass[11pt,english]{amsart}
\usepackage{dsfont,amssymb,amsmath,mathrsfs,fourier,baskervald}

\usepackage[text={6.5in,8.5in},centering]{geometry}

\usepackage{enumitem}
\usepackage[dvipsnames]{xcolor}   
\usepackage{xparse}
\usepackage{xr-hyper}
\usepackage[linktocpage=true,colorlinks=true,hyperindex,citecolor=teal,linkcolor=blue]{hyperref}  
\usepackage{tikz-cd}
\usepackage[all]{xy}
\usepackage{tikz-cd}
\newdir{ >}{!/6pt/@{ }*:(1,0.2)@_{>}*:(1,-0.2)@^{>}}
\usepackage{mathtools}

\newtheorem{theorem}{Theorem}[section]
\newtheorem{proposition}[theorem]{Proposition}
\newtheorem{lemma}[theorem]{Lemma}
\newtheorem{corollary}[theorem]{Corollary}
\theoremstyle{definition}
\newtheorem{definition}{Definition}[section]

\numberwithin{equation}{section}
\theoremstyle{remark}
\newtheorem{remark}[theorem]{Remark}

\def\leq{\leqslant}

\newcommand\twoheaddownarrow{\mathrel{\rotatebox[origin=c]{90}{$\twoheadleftarrow$}}}

\begin{document}
	
	\title{On hollowness in multiplicative lattices}
	
	\author{Amartya Goswami}
	\address{A. Goswami: [1]  Department of Mathematics and Applied Mathematics, University of Johannesburg, P.O. Box 524, Auckland Park 2006, South Africa. [2]  National Institute for Theoretical and Computational Sciences (NITheCS), South Africa.}
	\email{agoswami@uj.ac.za}
	
	\author{Joseph Zelezniak}
	
	\address{J. Zelezniak: Department of Mathematics and Applied Mathematics, University of Cape Town, Rondebosch 7700, South Africa}
	\email{ZLZJOS001@myuct.ac.za}
	
	%06F05 Ordered semigroups and monoids
	
	%06F10 Noether lattices
	
	%13A15 Ideals and multiplicative ideal theory in commutative rings
	
	
	\begin{abstract}
		The aim of this article is to extend the notions of strongly hollow and completely strongly hollow ideals of commutative rings to multiplicative lattices.  We investigate their basic structural properties and prove several characterizations in terms of localizations at maximal elements and the behaviour of residuals. In particular, we study properties of strongly hollow elements in various types of
		$C$-lattices: Gelfand, semi-simple, and Pr\"ufer. We also provide characterizations of quasi-local weak $r$-lattices by completely strongly hollow elements. Furthermore, we give characterization of strongly hollow elements in Noether lattices, and obtain explicit descriptions in Pr\"ufer and 
		$r$-lattices. Using strongly hollow and completely strongly hollow elements, we obtain representabilities of multiplicative lattices.
	\end{abstract}
	
	\maketitle 
	
	\section{Introduction}

	A \emph{multiplicative lattice}  is a complete lattice $(Q, \leqslant, 0, 1)$ endowed with an associative,
	commutative multiplication (denoted by $\cdot$), which distributes over arbitrary joins and has $1$ as multiplicative
	identity.
	In this paper, by $Q$, we will denote a multiplicative lattice. Let us recall a few  definitions from \cite{And-ML-1974,Dil62, And-ACIT-1976}.  An element $c$ of $Q$ is called \emph{compact} if for any $\{x_{\lambda}\}_{\lambda \in \Lambda}\subseteq Q$ and $c\leqslant \bigvee_{\lambda \in \Lambda} x_{\lambda}$ implies that $c\leqslant \bigvee_{i=1}^n x_{\lambda_i}$, for some $n\in \mathds{N}^+$, and  $Q$ is said to be \emph{compactly generated} if every element of
	$Q$ is the join of compact elements of $Q$. An element  $x$ of $Q$ is said to be \emph{finitely generated} if it is the join of a finite number of compact elements. A subset $S \subseteq Q$ is called \emph{multiplicatively closed}  if $1\in S$ and
	$xy \in S$ for each $x$, $y \in S$.  A multiplicative lattice $Q$ is called a \emph{$C$-lattice} if the set of compact elements $Q^{*}$ is multiplicatively closed, contains $1$, and every element of $Q$ is a join of compact elements. In any $C$-lattice, there exist minimal primes above an element (see \cite[p.~2]{Abs-CIT}). A multiplicative lattice $Q$ which is a C-lattice and also principally generated is called a \emph{weak $r$-lattice}.
	
	An element $a\in Q$ is called \emph{uniform} if the meet of any two elements less than it is non-zero. An element $a$ in a lattice $Q$ is called \emph{neutral} if $(a\wedge x)\vee (x\wedge y)\vee (y\wedge a) = (a\vee x)\wedge (x\vee y)\wedge (y\vee a)$, for all $x$, $y\in Q$. An element $p \in Q$ with $p < 1$ is \emph{prime} if $xy \leq p$ implies $x \leq p$ or $y \leq p$; the set of prime elements of $Q$ is denoted $\operatorname{Spec}(Q)$.  A prime is called a \emph{minimal prime} if there is no prime strictly less than it; the set of minimal primes is denoted by $\mathcal{P}_o$. A multiplicative lattice $Q$ is an \emph{integral domain} (or simply, a \emph{domain}) if $0$ is a prime element. An element $n \in Q$ is called \emph{nilpotent} if $n^{m} = 0$ for some positive integer $m$; the \emph{nilradical} of $Q$, denoted $\sqrt{0}$, is the join of all nilpotent elements of $Q$, equivalently the meet of all prime elements of $Q$. A multiplicative lattice $Q$ is called \emph{reduced} if there exist no nilpotent elements.  For $a$, $b \in Q$, the \emph{residual of $a$ by} $b$ is defined as $(a : b) = \bigvee \{ x \in L \mid x b \leq a \}$. An element $m\in Q$ is called \emph{maximal} if there exists no element strictly greater than it and strictly less than 1. We denote the set of all maximal elements in $Q$ by $\mathcal{M}$. By $J(Q):=\bigwedge \{ m\mid m\in \mathcal{M}\}$, we denote the Jacobson radical of $Q$. A multiplicative lattice $Q$ is called \emph{semi-simple} if  $J(Q)$ is zero. We note that all maximal elements are prime elements, and if $1$ is a compact element, then every element is less than a maximal element (see \cite[Lemma~2.3(1)]{Z-Lat}). 
    
    The \emph{localization} of $Q$ at a multiplicatively closed subset $S \subseteq Q^{*}$ containing $1$ is $Q_{S}=\{a_S\mid a\in Q\}$ where $a_{S}:=\bigvee\{x\in Q\mid  x\cdot s\leq a,\text{ for some } s\in S\}$. Note that $Q_S$ is a multiplicative lattice with the same order as $Q$ and the multiplication defined by $a_{S} \cdot b_{S}:= (ab)_{S}$, where the multiplication on the right is being evaluated in $Q$. If $q\in Q$ is a prime element then we set $S:=\{s\in Q^*\mid s\nleq q\}$, and refer to the localization $Q_S$ as $Q_q$. Observe that in a $C$-lattice we have that $a=b$ if and only if $a_m=b_m$, for all maximal elements $m\in Q$. We refer the reader to \cite{And-ML-1974} and \cite{Dil62} for more information on localization. A multiplicative lattice $Q$ is said to be \emph{quasi-local} if $1$ is compact and $Q$ has a unique maximal element.
	
	The notions of principal and weakly principal elements we use here are those from \cite{AJ-PE-1996} which were adapted from \cite{Dil62}. An element $e\in Q$ is said to be \emph{principal} if it satisfies  
	\begin{enumerate}
		\item $a\wedge be = ((a:e)\wedge b)e$, 
		
		\item $a\vee(b:e) =((ae\vee b):e)$,
	\end{enumerate}
	for all $a$, $b\in Q$.
	An element $e\in Q$ is called \emph{meet-principal} if it satisfies (1), and \emph{join-principal} if it satisfies (2).
	An element $a\in Q$ is \emph{weak principal}  if it satisfies 
	\begin{enumerate}
		\item[(a)] $e\wedge a = (a:e)e$,
		
		\item[(b)] $a\vee(0:e)=(ae:e)$,
	\end{enumerate}
	for all  $a\in Q$,
	and is called respectively \emph{weak join principal} and \emph{weak meet principal} if it satisfies (a) and (b) respectively.
	We will make heavy use of the characterizations of weak join principal and weak meet principal elements given by \cite[Lemma~1]{AJ-PE} as well as the \cite[Theorem~1]{AJ-PE}.
	
	We consider the standard Zariski topology on $\operatorname{Spec}(Q)$, where $V(a)=\{p\in\operatorname{Spec}(Q)\mid a\leq p\}$ (where $a\in Q$)
	is a base for the closed sets. We consider $\mathcal{M}$ as a subspace of $\operatorname{Spec}(Q)$. We note that since our multiplicative lattices have a top element $1$ which is a multiplicative identity, whenever $1$ is compact (for example, if $Q$ is a $C$-lattice), both $\operatorname{Spec}(Q)$ and $\mathcal{M}$ are compact topological spaces. We refer the reader to \cite{Facchini-AbsSpec-2021} for further studies of the Zariski topology on multiplicative lattices.  
	
	As a generalization of strongly hollow ideals, in \cite{Ros21}, Rostami has introduced the notion of completely strongly hollow ideals and studied various related properties. This paper aims to extend both these notions to multiplicative lattices. 
	
	Let us now summarize the main results in this paper. After introducing the notions of strongly hollow and completely strongly hollow elements in Section~\ref{bp}, we begin by examining their relationship with maximal elements. In particular, we show that any strongly hollow element avoids at most one maximal element (Theorem~\ref{T2.8}), and that the existence of a strongly hollow cancellation element forces locality (Proposition~\ref{P2.9}). 
	We provide a characterization of completely strongly hollow elements in terms of their associated completely irreducible elements (Theorems~\ref{T2.7}, \ref{T2.25}, and \ref{T2.26}). 
	
	In Section~\ref{hsl}, we characterize quasi-local weak $r$-lattices in terms of completely strongly hollow elements. In Propositions~\ref{1}–\ref{3}, Lemma~\ref{2}, and Theorem~\ref{4}, we establish that quasi-locality is equivalent to the property that every principal completely strongly hollow element is comparable with all other elements. We then turn to the existence and behaviour of strongly hollow elements in specific classes of lattices. In Proposition~\ref{P3.1}, Theorem~\ref{T3.4}, and Corollary~\ref{C3.6}, we show that in semi-simple Gelfand $C$-lattices the notions of strongly hollow, minimal, and uniform elements coincide. The section concludes with necessary and sufficient conditions for the existence of a completely strongly hollow element in a reduced Pr\"ufer $C$-lattice (Theorem~\ref{T3.8}).
	
	In Section~\ref{dml}, we turn to decompositions of elements and lattices using completely strongly hollow elements. Lemmas~\ref{L4.5} and \ref{L4.6} establish the finiteness of maximal and minimal elements in lattices that admit such representations.
	Finally, Section~\ref{cwn} studies completely strongly hollow elements in $r$-lattices. We show that if every element is a product of completely strongly hollow elements, then the lattice is a weak Noether chain lattice (Proposition~\ref{P5.3}). Furthermore, Theorem~\ref{T5.5} and Proposition~\ref{P5.6} establish equivalences between various classes of $r$-lattices and those in which every prime element is completely strongly hollow.
	
	\section{Basic Properties}
	\label{bp}
	
	\begin{definition}
		An element $a \in Q$ is called \emph{strongly hollow} if for $j$, $k\in Q$ and $a \leq j \vee k$ implies that $ a \leq j$ or $a \leq k$. An element $a \in Q$ is called \emph{completely strongly hollow} if for an arbitrary indexed family $\{i_\lambda\}_{\lambda \in \Lambda}\subseteq Q$ and $a \leq \bigvee_{\lambda \in \Lambda} i_\lambda$ implies that $a \leq i_{\lambda'}$, for some
		$\lambda' \in \Lambda$.
	\end{definition}
	
	\begin{remark}
		It is worth mentioning that strongly hollow elements are studied in other works, such as \cite{Abs-CIT}, where they are called `join-prime' elements of a lattice.
	\end{remark}
	
	Recall that in the category $\mathds{ML}\text{at}$  of multiplicative lattices, an adjunction $(f,u):X\rightarrow Y$ is a pair of order preserving maps:
	\[
	\begin{tikzcd}[column sep=4em] 
		X
		\arrow[r, shift left=.8ex, "f"] 
		& Y,
		\arrow[l, shift left=.8ex, swap, "u" below]
	\end{tikzcd}
	\]
	with $f(x)\leq y \Leftrightarrow x\leq u(y)$. 
	
	\begin{proposition}
		Let $(f,u):X\rightarrow Y$ be a morphism in $\mathds{ML}\text{at}$ and $u$ preserves binary joins (e.g. is a left adjoint itself ). If $x$ is a strongly hollow element in $X$, then $f(x)$ is strongly hollow in $Y$.
	\end{proposition}
	
	\begin{proof}
		Suppose that $x$ is strongly hollow in $X$ and $f(x)\leq r\vee s$. Then we have that $x\leq u(r\vee s)=u(r)\vee u(s)$. Therefore, $x\leq u(r)$ or $x\leq u(s)$ giving $f(x)\leq r$ or $f(x)\leq s$. 
	\end{proof}
	
	\begin{remark}
		If $u$ is completely join-preserving, that is, it is a left adjoint, then the same argument above shows that $f$ preserves completely strongly hollow elements.	
	\end{remark}
	In the case of finitely generated elements, the two types of hollowness coincide.
	
	\begin{proposition}
		A finitely generated element $a \in Q$ is completely strongly hollow if and only if it is strongly hollow.
	\end{proposition}
	
	\begin{proof}
		If $a$ is completely strongly hollow, then by definition, it is strongly hollow when restricted to finite joins.
		For the converse, suppose $a$ is a finitely generated element of $Q$ and $a$ is strongly hollow. Then $a$ is compact (see \cite[Proposition~2.1]{Cal}). Assume that $a \leq \bigvee_{\lambda \in \Omega} i_\lambda$, then there exists a finite subset $S \subset \Omega$ such that $a \leq \bigvee_{\lambda \in S} i_\lambda$. Since $a$ is strongly hollow, there exists $\lambda \in S$ such that $m \leq i_\lambda$. Thus, $a$ is completely strongly hollow.
	\end{proof}
	
	A particularly important characteristic of completely strongly hollow elements is that if a multiplicative lattice $Q$ is generated by a set of elements \(S\subset Q\), then any completely hollow element \(a\in Q\) is a member of the generating set, that is, \(a\in S\). This means that in any multiplicative lattice which is generated by principal (weak principal), compact, or prime elements, any completely strongly hollow element has the property in question. Moreover, if the generating set is finite, the same statement holds for strongly hollow elements. Consequently, in any multiplicative lattice generated by principal, compact, or prime elements, any subsequent result that assumes a completely strongly hollow element is a member of the generating set holds for all completely strongly hollow elements.

	\begin{proposition}
		If an element $0\neq a \in Q$ is completely strongly hollow, then the set $\{ j \mid j < a \}$ has exactly one maximal element.
	\end{proposition}
	
	\begin{proof}
		Let $A := \{ j \mid j < a \}$. Since $0 < a$, we have $A \neq \emptyset$. Now $\bigvee A \leq a$, and suppose for contradiction that $\bigvee A = a$. Then $a \leq \bigvee A$, which implies $a \leq j$, for some $j \in A$, a contradiction. Thus, $\bigvee A$ is the unique maximal element.
	\end{proof}
	
	\begin{proposition}\label{T2.3}
		If $i$, $j$ are strongly hollow elements in $Q$, then $i \vee j$ is strongly hollow if and only if either $i \leq j$ or $j \leq i$.
	\end{proposition}
	
	\begin{proof}
		If either $i \leq j$ or $j \leq i$, the result is trivial.
		Suppose $i \vee j$ is strongly hollow, and apply the definition to $i \vee j \leq i \vee j$, which gives $i \vee j \leq i$ or $i \vee j \leq j$. Then $i \vee j = i$ or $i \vee j = j$, so $i \leq j$ or $j \leq i$.
	\end{proof}
	
	\begin{proposition}
		Let $a (\neq 0)$ be a strongly hollow element of $Q$. Then the following hold:
		\begin{enumerate}
			\item If $a = r \vee s$, then either $r \leq s$ or $s \leq r$.
			\item For $r \in Q$, either $a \leq r$ or $a \leq s^+$, where $s^+ := \bigwedge \{ s \mid s \vee r \geqslant a \}$.
		\end{enumerate}
	\end{proposition}
	
	\begin{proof}
		(1) Suppose $a$ is  strongly hollow and $a = r \vee s$. Then $a \leq r$ or $a \leq s$. Thus, if $a \leq r$, then $a = r \vee s \leq r$ implies that $s \leq r$, and the other implication follows in the same way when supposing $a \leq s$. 
		
		(2) For ease of notation take $S := \{s \mid s \vee r \geq a\}$. We have that $a \leq r \vee s$, for all $s \in S$. If $a \leq r$, then we are done; otherwise $a \leq s$ for all $s \in S$, and thus $a \leq \bigwedge S = s^+$.
	\end{proof}
	
	\begin{proposition}\label{T2.8}
		Let $0 \neq a \in Q$. If $a$ is strongly hollow, then either $a \leq J(Q)$ or there exists a unique maximal element of $Q$ which is not greater than $a$.
	\end{proposition}
	
	\begin{proof}
		Suppose $a \nleq J(Q)$, then there exists a maximal element $\mathfrak{m} \ngeq a$. Now if  $\mathfrak{n}$ is another maximal element, then $a \leq \mathfrak{m} \vee \mathfrak{n} = 1$. Hence $a \leq \mathfrak{n}$.
	\end{proof}
	
	\begin{proposition}\label{P2.9}
		If $Q$ has a strongly hollow element which is also a cancellation element, then $Q$ has at most one maximal element.
	\end{proposition}
	
	\begin{proof}
		Let $a \in Q$ be strongly hollow as well as a cancellation element. Suppose there exist two maximal elements $\mathfrak{m}$ and $\mathfrak{n}$. Then $a \leq \mathfrak{m} \vee \mathfrak{n} = 1$, and thus, \[a \cdot \mathfrak{m} \vee a \cdot \mathfrak{n} = a \cdot (\mathfrak{m} \vee \mathfrak{n}) = a.\] 
		Without loss of generality, suppose by the strongly hollow property of $a$ that $a \leq a \cdot \mathfrak{m}$. Then, since $a \cdot \mathfrak{m} \leq a \wedge \mathfrak{m} \leq a$, we have that $a = a \cdot \mathfrak{m}$. Thus, since $a$ is a cancellation element, $\mathfrak{m} = 1$, a contradiction. Hence, $Q$ has at most one maximal element.
	\end{proof}
	
	The next two results are on the behaviour of hollow elements in quasi-local lattices.
	
	\begin{proposition}\label{P2.7}
		Let (Q,m) be a quasi-local lattice. Then m is a strongly hollow element if it is weak principal.
	\end{proposition}
	
	\begin{proof}
		Suppose that $m$ is weak principal and $m\leq r \vee s$. If either $r$ or $s$ equals $1$, then we are done. Thus, assume that $m=r\vee s$. Then, since $s,r\leq m$  and since $m$ is weak meet principal, there exist $c,d \in Q$ with $r=cm$ and $s=dm$. Hence, \[m=cm\vee dm = m(c\vee d),\] and now, since $m$ is weak join principal, we have that $1=(c\vee d)\vee (0:m)$. If $(0:m)=1$, then we have that $m=0$, which implies $Q$ is a trivial finite lattice and $m$ is completely strongly hollow. If $c\vee d =1$, then either $c=1$ or $d=1$ giving $m=r$ or $m=s$, leads to the desired conclusion. 
	\end{proof}
	
	\begin{proposition}\label{P2.11}
		Let $1\in Q$ be compact. Then $1$ is completely strongly hollow if and only if $Q$ is a quasi-local lattice.
	\end{proposition}
	
	\begin{proof}
		The `if' direction is trivial since $1$ is compact and strongly hollow (since every element is below the unique maximal element). Notice that if $1$ is strongly hollow and $m,n\in Q$ are two distinct maximal elements, then $1\leq m\vee n$, implying that $1=m \text{ or } 1=n$, which is a contradiction. Therefore, $Q$ is quasi-local.
	\end{proof}

	We next aim to show that every completely strongly hollow element is associated with a completely strongly irreducible element. We also show that the property of being strongly hollow is preserved by passage to a quotient lattice and under certain localizations. 
	Let us first introduce a few new notations.

	\begin{definition}
		For an element $a \in Q$ we set $T(a) := \{k \mid a \nleq k \}$ and $\kappa(a) := \bigvee T(a)$, and we denote $L_a=(\kappa(a):a)$.
	\end{definition}
	
	\begin{theorem}\label{T2.7}
		Let $a$ be a non-zero compact element of $Q$. Then $a$ is strongly hollow (and thus, completely strongly hollow) if and only if there exists the greatest element of $Q$ with respect to not being greater than $a$. In the case where $a$ is completely strongly hollow, this element is $\kappa(a)$.
	\end{theorem}
    
	\begin{proof}
		Suppose $a$ is a non-zero strongly hollow compact element of $Q$. Then if $k \in T(a)$, we have $k \leq \kappa(a)$. If $a \leq \kappa(a)$ then, since $a$ is completely strongly hollow, we have $a \leq k$ for some $k \in T(a)$, a contradiction. Thus $a \nleq \kappa(a)$ and $\kappa(a)$ is the greatest element not greater than $a$.
		
		Now suppose there exists a greatest element $j$ with respect to not being greater than $a$. Consider an arbitrary family $\{i_\lambda\}_{\lambda \in \Omega}\subseteq Q$ such that $a \leq \bigvee_{\lambda \in \Omega} i_\lambda$. Assume, for contradiction, that $a \nleq i_\lambda$ for all $\lambda$. For fixed $\lambda$, define
		\[
		U_\lambda := \{ k \mid i_\lambda \leq k \text{ and } a \nleq k \}.
		\]
		We show $U_\lambda$ has a maximal element. Let $C \subseteq U_\lambda$ be a chain. Then $\bigvee C \in U_\lambda$ because if $a \leq \bigvee C$, then by compactness of $a$, $a \leq \bigvee F$ for some finite $F \subseteq C$. Since $F$ is a finite chain, $\bigvee F \in F$, so $a \leq k$ for some $k \in C$, a contradiction. Hence $\bigvee C \in U_\lambda$, so by Zorn’s Lemma, $U_\lambda$ has a maximal element $u_\lambda$.
		Now, each $u_\lambda$ is not greater than $a$ and maximal with this property, so $u_\lambda = j$. Thus $i_\lambda \leq j$ for all $\lambda$, and hence, $a \leq \bigvee i_\lambda \leq j$, contradicting $j \nleq a$. So, some $i_\lambda$ satisfies $a \leq i_\lambda$. Hence, $a$ is completely strongly hollow.
	\end{proof}
	
	\begin{proposition}
		Let $a$ be a non-zero compact element of $Q$. Then $a$ is completely strongly hollow if and only if $a \nleq \kappa(a)$.
	\end{proposition}
	
	\begin{proof}
		If $a$ is strongly hollow, then by Proposition \ref{T2.7}, $a\nleq\kappa(a)$.
		Conversely, suppose $a \nleq \kappa(a)$ and $a \leq r \vee s$ for some $r, s \in Q$. If $a \nleq r$ and $a \nleq s$, then $r, s \leq \kappa(a)$, and thus, $a \leq r \vee s \leq \kappa(a)$, contradiction. Therefore, $a$ is strongly hollow, and thus, completely strongly hollow (since $a$ is compact).
	\end{proof}
	
	\begin{corollary}
		If $a$ and $b$  are completely strongly hollow elements in $Q$ then, \[ a\leq b \Rightarrow \kappa(a)\leq\kappa(b).\]
	\end{corollary}
	
	\begin{proof}
		Suppose $a\leq b$ and $b\leq\kappa(a)$, then $a\leq b\leq \kappa(a)$, which is a contradiction. Then by the definition of $\kappa(b)$ and the fact that $b\nleq \kappa(a)$, we have  $\kappa(a)\leq\kappa(b)$. Now suppose $\kappa(a)\leq\kappa(b)$. Notice that if $a\nleq b$, then $b\leq \kappa(a)$, and thus, $b\leq \kappa(b)$, which is a contradiction, thus $a\leq b$.
	\end{proof}
	
	Our next goal is to study strongly hollow elements in quotient lattices.    Let $Q$ be a multiplicative lattice and $i\in Q$. According to \cite{Dil62}, a \emph{quotient lattice}, $Q/i$, is a complete sublattice of Q with a top $1$ and a bottom $i$, with elements corresponding to $Q/i:=\{x\in Q\mid x\geqslant i\}=:[i,1]$. It can be endowed with a multiplication (denoted by $\circ$) in order to make it a multiplicative lattice, \[a \circ b := (a\cdot b)\vee i \quad \forall a,b\in Q.\]
	$Q/i$ is closed under this multiplication, and residuals are equivalent in both $Q$ and $Q/i$. Furthermore, any element $b\in Q$ has a corresponding element $b\vee i \in Q/i$, we shall often refer to $b\vee i$ as $b/i$.
	For the rest of this paper, we shall denote by $Q/i$ the multiplicative lattice given by $(Q/i,\wedge, \vee,\circ, 1, i)$ for any element $i\in Q$.
	
	\begin{proposition}
		Let $a$ be a strongly hollow element of $Q$ and $i\in Q$. Then $a\vee i$ is strongly hollow in $Q/i$.
	\end{proposition}
	
	\begin{proof}
		If $a\leq i$, then $a/i =a\vee i = i$, which is the bottom element of $Q/i$, and thus, is strongly hollow, and we are done. Therefore, suppose  $a\nleq i$ and $a\vee i \leq r\vee s$ for some $r,s\in Q/i$. Then $a\leq r\vee s \Rightarrow a\leq r \; \text{or} \; a\leq s$. Without loss of generality, assume $a\leq r$. Since $r\in Q/i$, we have that $r\geqslant i$. Hence $a\vee i\leq r$.
	\end{proof}
	
	Now one can show directly that $a\vee i$ is compact in $Q/i$ and obtain an equivalence of the above proof for completely strongly hollow elements. We include the proof below to demonstrate the use of $\kappa(a)$. 
	
	\begin{proposition}
		Let $a$ be a completely strongly hollow element of $Q$ and $i\in Q$. Then $a\vee i$ is completely strongly hollow in $Q/i$.
	\end{proposition}
	
	\begin{proof}
		If $a\leq i$, then $a/i =a\vee i = i$, which is the bottom element of $Q/i$, and we are done. Thus, assume that $a\nleq i$. If $a\vee i \leq a$, then $a\vee i = a$, and thus $a\vee i$ is strongly hollow in the $Q/i$ since join and meet are the same in the sublattice as in the normal lattice. Thus, we can assume $a\vee i \nleq a$, and additionally notice that $a\vee i \nleq i$ because $a\nleq i $. Thus, suppose, for a contradiction, that \[a\vee i \leq \kappa_{Q/i}(a\vee i)= \bigvee T_{Q/i}(a\vee i) = \bigvee\{r\mid r\geqslant i \; \& \; a\vee i \nleq r\},\] and hence, $a\leq \kappa_{Q/i}(a\vee i)$. If $x\in T_{Q/i}(a\vee i)$, then $x\geqslant i$ and $x\ngeq a\vee i$, which implies that $x\ngeq a$, hence $x\in T(a)$. Then $T_{Q/i}(a\vee i) \subseteq T(a)$, and thus, $\kappa_{Q/i}(a\vee i) \leq \kappa(a)$. This gives that $a\leq \kappa(a)$, which is a contradiction. Therefore, $a\vee i\nleq\kappa_{Q/i}(a\vee i)$. Consider $a\vee i\leq \bigvee_{\alpha\in\Omega}r_\alpha$ for $r_\alpha \geqslant i$. Since $a$ is compact, there is a finite subset $F\subseteq \Omega$ such that $a\leq \bigvee_{\alpha\in F}r_\alpha$ ,which gives that $a\vee i\leq\bigvee_{\alpha \in F}r_\alpha$, and so $a\vee i$ is compact in $Q/i$. Therefore, $a\vee i$ is completely strongly hollow in $Q/i$.
	\end{proof}
	
	We next study the nature of elements of the form $L_a$, whenever $a$ is a strongly hollow element.
	
	\begin{theorem}
		Let $a\in Q$ be a strongly hollow element which is also a join principal. Then $L_a$ is the top element or a maximal element in Q.
	\end{theorem}
	
	\begin{proof}
		There are two cases: $a\leq \kappa(a)$ or $a\nleq\kappa(a)$. If $a\leq\kappa(a)$, then \(1\cdot a=a\leq\kappa(a)\) and thus $L_a=1$, and we are done. Thus, assume $a\nleq \kappa(a)$. Then $a\vee\kappa(a)\neq\kappa(a)=0_{Q/\kappa(a)}$ and $a\vee\kappa(a)$ is strongly hollow in $Q/\kappa(a)$. Suppose $r>\kappa(a)$ with $r\leq a\vee\kappa(a)$. If $a\leq r$, then $r\leq a\vee\kappa(a)\leq r\vee\kappa(a)=r$, and so, $r=a\vee\kappa(a)$. If $a\nleq r$, then $r\leq\kappa(a)$, which is a contradiction. Thus, $a\vee\kappa(a)$ is minimal in $Q/\kappa(a)$. Now consider  
		\begin{align*}
			(0:(a\vee\kappa(a)))_{Q/\kappa(a)}&=\bigvee\{x\geqslant\kappa(a)\mid  x\circ(a\vee\kappa(a))=\kappa(a)\} \\
			&=\bigvee\{x\vee\kappa(a)\mid x\in Q;\: (x\vee\kappa(a))\circ(a\vee\kappa(a))=\kappa(a)\} \\
			&=\bigvee\{x\vee\kappa(a)\mid x\in Q;\: [(x\vee\kappa(a))\cdot(a\vee\kappa(a))]\vee\kappa(a)=\kappa(a)\} \\
			&=\bigvee\{x\vee\kappa(a)\mid x\in Q;\: [x\cdot a\vee\kappa(a)\cdot a\vee x\cdot\kappa(a)\vee\kappa^2(a)]\vee\kappa(a)=\kappa(a)\} \\
			&=\bigvee\{x\vee\kappa(a)\mid x\in Q;\: x\cdot a\vee\kappa(a)=\kappa(a)\}\\
			&=\bigvee\{x\vee\kappa(a)\mid x\in Q;\: x\cdot a\leq\kappa(a)\}.   
		\end{align*}
		But since $x\circ( a\vee\kappa(a))\leq a\vee\kappa(a)$ and $a\vee\kappa(a)$ is minimal, we also have that
		\begin{align*}
			(0:(a\vee\kappa(a)))_{Q/\kappa(a)}&=\bigvee\{x\in Q\mid  (x\vee\kappa(a))\circ(a\vee\kappa(a))=\kappa(a)\}\\&=\bigvee\{x\in Q\mid  (x\vee\kappa(a))\circ(a\vee\kappa(a))\neq a\vee\kappa(a)\}\\
			&=\bigvee\{x\in Q\mid  (x\cdot a)\vee \kappa(a)\neq a\vee\kappa(a)\}.
		\end{align*}
		Let $l\geqslant (0:(a\vee\kappa(a)))_{Q/\kappa(a)}$. Then $l\circ (a\vee\kappa(a))=a\vee\kappa(a)$ and since $a\vee\kappa(a)$ is weak join principal in $Q/\kappa(a)$, we have $1\leq l\vee (0:(a\vee\kappa(a)))_{Q/\kappa(a)}=l$. Thus $(0:(a\vee\kappa(a)))_{Q/\kappa(a)}$ is maximal in Q. But notice that \[(0:(a\vee\kappa(a)))_{Q/\kappa(a)}=\bigvee\{x\vee\kappa(a)|x\in Q;\: x\cdot a\leq\kappa(a)\}=\kappa(a)\vee L_a=L_a,\] and thus, $L_a$ is maximal in $Q$.
	\end{proof}
	\begin{corollary}
		If $a$ is completely strongly hollow and join-principal in $Q$, then $L_a$ is a maximal element.
	\end{corollary}
	
	\begin{proposition}\label{T2.13}
		Let $Q$ be a $C$-lattice and $a\in Q$ a join principal, compact, and strongly hollow element. Then $a_{L_a}$ is strongly hollow in $Q_{L_a}$. 
	\end{proposition}
	
	\begin{proof}
		If $a_{L_a}=0_{L_a}$, then we are done. Thus, assume that $a_{L_a}\geqslant0_{L_a}$, and so, there exists at least one compact element $s\nleq L_a$. For all elements $s\in Q^*$ with $s\nleq L_a$, we have $a\cdot s\leq a$ which gives $a\leq a_{L_a}$. Suppose $a_{L_a}\leq (\kappa(a))_{L_a}$. Then $a\leq(\kappa(a))_{L_a}$, and since $a$ is completely strongly hollow, we have $a\leq x$, for some $x\in Q$ with $x\cdot s\leq\kappa(a)$, for some $s\in Q^*$ with $s\nleq L_a$. Thus $a\cdot s\leq x\cdot s\leq\kappa(a)$ which contradicts $s\nleq L_a$. Therefore, $a_{L_a}\nleq(\kappa(a))_{L_a}$. Now, let $l\in Q$  be such that $l_{L_a}\ngeq a_{L_a}$. Then we have that $l\ngeq a$, and thus $l\leq \kappa(a)\leq (\kappa(a))_{L_a}$, and hence, $(\kappa(a))$ is the greatest element with respect to not being greater than $a_{L_a}$. Since $a_{L_a}$ is compact in $Q_{L_a}$, therefore, we have that $a_{L_a}$ is completely strongly hollow by Theorem \ref{T2.7}.
	\end{proof}
	
	\begin{proposition}\label{T2.14}
		Let $Q$ be a multiplicative lattice with $1$ compact and $a\in Q$ strongly hollow, compact, and join principal. Then $Q/(0:a)$ is a quasi-local lattice with a unique maximal element, $L_a/(0:a)$.
	\end{proposition}
	
	\begin{proof}
		We have that $a\vee (0:a)$ is strongly hollow in $Q/(0:a)$. Suppose there exists $m\in Q/(0:a)$ such that $(a\vee (0:a))\cdot m=a\vee (0:a)$. Since $a\vee (0:a)$ is weak join principal in $Q/(0:a)$, we have \[1\leq m\vee(0:(a\vee (0:a)))_{Q/(0:a)}.\] Now, observe that 
		\begin{align*}
			(0:(a\vee (0:a)))_{Q/(0:a)}&=\bigvee\left\{x\vee (0:a) \left | (x\vee (o:a))\circ(a\vee (0:a))=(0:a)\right. \right\} \\
			&=\bigvee\left\{x\vee (0:a) \left |(x\cdot a\vee x\cdot (0:a)\vee (0:a)=(0:a)\right. \right\} \\
			&=\bigvee\left\{x\vee (0:a) \left |(x\cdot a\leq (0:a)\right. \right\} \\
			&= (0:a)\vee((0:a):a)\\
			&= ((0:a):a).
		\end{align*}
		Thus $(0:(a\vee (0:a)))_{Q/(0:a)}=(0:a^2)$ by \cite[Eqn.~2.5]{Dil62}. Since \(1=m\vee(0:a^2)\), we have that \(a=am\vee a(0:a^2)\). Since $a$ is strongly hollow, we have \[a=am\quad\text{or}\quad a=a(0:a^2).\] If $a=am$ for all such $m$, then since $a$ is join principal, we have that $1\leqslant m \vee (0:a)=m$ since $m\geqslant (0:a)$, and thus, since $a\vee (0:a)$ is a cancellation element, we have that $Q/(0:a)$ is a quasi-local lattice by Proposition \ref{P2.9}. Since $\kappa(a)\geqslant 0=a\cdot (0:a)$, we have $(0:a)\leq L_a$; and since $L_a$ is maximal, it must be the unique maximal element in $Q/(0:a)$. Otherwise, assume there exists $m\geqslant (0:a)$ such that $(a\vee (0:a))\cdot m=a\vee (0:a)$ and $a\neq am$. Then $a=a\cdot(0:a^2)$. Now, since $a\cdot(0:a^2)=\bigvee\left\{a\cdot x\mid x\cdot a^2=0 \right\}$, we have $a=a\cdot(0:a^2)\leq (0:a)$. Moreover, $L_a$ is maximal and greater than (0:a). Suppose there exists another maximal element $n\geqslant (0:a)$, and thus, $L_a\vee n=1$ implying $(a\cdot L_a)\vee(a\cdot n)=a$. Thus, since $a$ is strongly hollow, we have that $a=a\cdot L_a$ or $a=a\cdot n$ but $a\cdot L_a\leq \kappa(a)$, and thus, $a\neq a\cdot L_a$. Therefore, $a=a\cdot n$. But then, since $a$ is join principal, we have that $1\leq n\vee (0:a)=n$. So, $L_a$ is the unique maximal element in $Q/(0:a)$, and thus, it is a quasi-local lattice.
	\end{proof}
	
	\begin{proposition}\label{T2.15}
		Let a be a non-zero element of a C-lattice Q such that a is compact and join principal. Then a is strongly hollow if and only if the lattice $Q/(0:a)$ is a quasi-local lattice and either $a_m=0_m$ or $a_m$ is strongly hollow in $Q_m$ for every maximal element m.
	\end{proposition}
	
	\begin{proof}
		If $a$ is strongly hollow, then by Theorem \ref{T2.14}, $Q/(0:a)$ is a quasi-local lattice with the unique maximal element $L_a$. Let $m$ be a maximal element of $Q$. If $m\geqslant (0:a)$, then $m=L_a$ and we have that $Q_m=Q_{L_a}$, and thus, by Proposition \ref{T2.13}, we have the desired claim. Now, suppose $m\ngeq (0:a)$. Since we are in a $C$-lattice, there exists a family of compact elements $\{c_i\}_{i\in \Omega}$  such that \((0:a)=\bigvee_{ i \in \Omega }c_i \). Since $0=a\cdot (0:a)=\bigvee_{ i \in \Omega }a\cdot c_i$, we must have that $a\cdot c_i=0$, for all $i\in \Omega$. Notice, if $m\geq c_i$ for all i, then we have that $m\geqslant (0:a)$, thus, there must exists an $i\in \Omega$ such that $c_i\nleq m$, and thus, \[c_i\cdot a=0 \Rightarrow (c_i\cdot a)_m=0_m \Rightarrow (c_i)_m\cdot a_m=0_m.\] Since \(c_i\nleq m\), we have that \((c_i)_m=1_m\) giving \[1_m\cdot a_m=0_m\Rightarrow a_m=0_m.\]
		Now, suppose $Q/(0:a)$ is a quasi-local lattice with unique maximal element $n$ and $a_m=0_m$ or $a_m$ is strongly hollow in $Q_m$ for all maximal elements $m\in Q$. First notice that if $m\ngeq (0:a)$, then $a_m=0_m$, and thus, if $a_n=0_n$, we would have that $a=0$, and we are done. Thus, suppose $a_n\neq 0_n$, and thus, $a_n$ strongly hollow in $Q_n$. Suppose $a\leq r\vee s$ for $r,s\in Q$ which gives $a_n\leq r_n\vee s_n$. Suppose without loss of generality that $a_n\leq r_n$. Now, $a_m\leq r_m$ for all maximal elements $m$ in $Q$. Therefore, $a\leq r$, and thus, $a$ is strongly hollow.
	\end{proof}
	
	The elements that are neutral and complemented have a close connection with strongly hollowness, which we shall observe now.
	
	\begin{proposition}
		Let $a$ be a neutral and complemented element of $Q$. Then there exists at most one element maximal with respect to being less than $a$ if $a$ is strongly hollow.
	\end{proposition}
    
	\begin{proof}
		Suppose $a$ is strongly hollow. Since $a$ is neutral and complemented, it defines an embedding of $Q$ by $\varphi(x)=\langle x\vee a,x\wedge a\rangle$. Assume there exists $m,n$ maximal in $[0,a]$. Then, consider $m\vee b$ and $n\vee b$, where $b$ is a complement of $a$. If $x\geqslant m\vee b$ then $x\wedge a \geqslant (m\vee b)\wedge a$. Since $a$ is neutral, it must be distributive, and thus, we have that \[x\wedge a\geqslant (m\wedge a)\vee(b\wedge a)=m\wedge a=m.\] Notice that, if $x\geqslant a$ then $x\geqslant a\vee b=1$. Otherwise, if $a\nleq x$, we have $m\leqslant x\wedge a<a$ implies $x \wedge a= m $. Therefore, \(a\vee x\geqslant  a\vee  b=1=a\vee (m\vee b)\) and $a\wedge(m\vee b)=m=a\wedge x$. Since $a$ is neutral this shows that $x=m\vee b$. Therefore, $m\vee b$ and $n\vee b$ are maximal in $Q$. But we have  $(m\vee b)\vee a = 1$ and $(n\vee b)\vee a=1$. Therefore, $a\nleq m\vee b$ and $a\nleq m\vee n$. Hence, $m\vee b$ and $n\vee b$ are two maximal elements of $Q$ which are not greater than $a$, a contradiction to Proposition \ref{T2.8}. Thus, $[0,a]$ has a unique maximal element.
	\end{proof}

	Note that the converse can be shown to hold under the additional requirement that $a$ be compact. 
	\begin{proposition}
		Let $a$ be a neutral, complemented, and compact element of $Q$. If there exists a unique element maximal with respect to being less than $a$, then $a$ is strongly hollow.
	\end{proposition}
	
	\begin{proof}
		Suppose $[0,a]$ has a unique maximal element, and denote that element by $m$. Suppose $a\leq r\vee s$. Then $r=u\vee (v\wedge b)$ and $s=u'\vee(v'\wedge b)$ with $u,u'\leq a$ and $v,v'\geqslant a$. Thus,
		\begin{align*}
			a&\leq (u\vee u')\vee\left((v\wedge b)\vee(v'\wedge b) \right)\\
			&\leq((u\vee u')\vee\left((v\wedge b)\vee(v'\wedge b) \right))\wedge a \\
			&\leq (u\vee u')\vee((v\wedge b\wedge a)\vee(v'\wedge b\wedge a)) \\
			&\leq(u\vee u').
		\end{align*}
		Thus, $a=u\vee u'$. Now, since $a$ is compact and is the top element of $[0,a]$, every proper element in \([0,a]\) is contained in a maximal element (see \cite[Lemma 2.3(2)]{Z-Lat}). Thus,  if  both $u$ and$u'$ are proper elements, we have $u\vee u'\leq m<a$. Thus, either $u=a$ or $u'=a$. Without loss of generality, assume that $u=a$. Then, we have that $r\geqslant u =a$.
	\end{proof}
    \begin{corollary}\label{T2.18}
        Let $a$ be a compact, neutral, and complemented element of $Q$. Then there exists a unique  element  maximal with respect to being less than $a$ if and only if  $a$ is strongly hollow.
    \end{corollary}
	\begin{theorem}\label{T2.20}
		Let $1$ be compact in Q. Suppose $0\neq a\in Q$ is neutral and complemented. Then $a$ is strongly hollow if and only if $x<a$ implies $x\leq J(Q)$ for all $x\in Q$.
	\end{theorem}
	
	\begin{proof}
		Suppose $a$ is strongly hollow. If $a\leq J(Q)$, we are done. Suppose $a\nleq J(Q)$. By Proposition \ref{T2.8}, there exists a unique maximal element $m\in Q$ such that $a\nleq m$. Let $b$ denote the complement of $a$. Let $n$ be a maximal element such that $n\neq m$. Since $a\leq n$, we have that $1= a\vee b \leq n\vee b$. Thus $b\nleq n$, for all maximal elements $n\neq m$. Since $1$ is compact, $b$ must lie below a maximal element. Thus $b\leq m$. Now, let $x<a$. Suppose $x\nleq m$, then $x\vee m=1$. Consider \[a= a\wedge(x\vee m)= (x\wedge a)\vee(m\wedge a)=x\vee(m\wedge a) \Rightarrow a\leq x \: \text{or}\: a\leq a\wedge m.\]
		But $a\nleq x$ since $x<a$ and $m\geqslant a\wedge m=a$, which is a contradiction. Therefore, $x\leq m$ and thus $x\leq J(Q)$.
		
		Now, assume $x<a$ implies $ x\leq J(Q)$, for all $x\in Q$. Suppose $a\leq r\vee s$, for $r,s\in Q$. Then, we have that $a^2\leq a(r\vee s)=ar\vee as$. But  \[a\vee b =1 \Rightarrow a(a\vee b)=a \Rightarrow a^2\vee ab =a,\]
		and $ab\leq a\wedge b=0$; and thus $a^2=a$. Therefore, $a=ar\vee as$. If $a\leq ar$ or $a\leq as$, then we are done. Thus, assume that $ar<a$ and $as<a$. Then $ar\vee as\leq J(Q)$, and thus $a\leq J(Q)$. Let $m$ be a maximal element of $Q$. Thus $a\leq m$, which gives that $b\nleq m$. But then $b$ is not less than any maximal element, which is a contradiction. Therefore, $a=ar\leq r$ or $a=as\leq s$.
	\end{proof}
	
	The next two theorems can be considered as the penultimate results on various relations between $\kappa(a)$, $L_a$, and strongly hollow elements.
	
	\begin{theorem}\label{T2.25}
		Let $Q$ be a $C$-lattice with $1$ compact. Let $a$ be a non-zero, compact, weak join principal, strongly hollow element of $Q$. Then the following conditions are equivalent,
		\begin{enumerate}
			\item $\kappa(a)$ is prime,
			\item $a^2\nleq \kappa(a)$ ,
			\item $a=a^2$, 
			\item $a\nleq J(Q)$,
			\item $\kappa(a)$ is maximal and $\kappa(a)=L_a$.
		\end{enumerate}
	\end{theorem}
	
	\begin{proof}
		$(1)\Rightarrow (2)$: Observe that $a^2\leq \kappa(a)$ would imply $ a\leq\kappa(a)$, which is a contradiction. Thus, $a^2\nleq \kappa(a)$. 
		
		\((2)\Rightarrow(3)\): If $a^2\nleq \kappa(a)$, then $a^2\geqslant a$, therefore $a^2=a$. 
		
		$(3)\Rightarrow(4)$: Suppose $a\leq J(Q)$. Let $m$ be maximal in $Q$. We have $a\leq m$, thus $a=a^2\leq am\leq a$, giving $a=am$. Since $a$ is weak join principal, we have that $1\vee(0:a)=m\vee(0:a)$ giving $1=m\vee(0:a)$. Thus, $(0:a)$ is not less than any maximal element, and it must therefore be equal to $1$. If $(0:a)=1$, then  we have $a=0$ (since $a\cdot (0:a)=0$), which is a contradiction. Thus $a\nleq J(Q)$. 
		
		\((4)\Rightarrow(5)\): Suppose $a\nleq J(Q)$. Let $m$ be the unique maximal element which is not greater than $a$. Thus $m\leq\kappa(a)$. If $\kappa(a)=1$, then we have a contradiction. Thus $\kappa(a)=m$ and is maximal. If $L_a=1$, then we have that $a=a\cdot L_a\leq\kappa(a)$, which is a contradiction. Since $\kappa(a)\leq L_a$, therefore, we have $L_a=\kappa(a)$.
		
		\((5)\Rightarrow(1)\): Since all maximal elements in a multiplicative lattice are prime, $\kappa(a)$ is prime.
	\end{proof}
	
	\begin{theorem}\label{T2.26}
		Let a be a non-zero completely strongly hollow element of Q. Then $\kappa(a)$ is completely strongly irreducible.
	\end{theorem}
	
	\begin{proof}
		Consider \( \{r_i\}_{i\in \Omega}\subset Q \) such that \(\bigwedge_{i\in \Omega}r_i\leq \kappa(a)\) and $r_i\nleq \kappa(a)$, for all $i\in\Omega$. Then $a\leq r_i$, for all $i\in\Omega$, and thus \[a\leq \bigwedge_{i\in \Omega}r_i\leq \kappa(a), \] which is a contradiction. Thus, there must exists $r_i$ such that $r_i\leq \kappa(a)$.
	\end{proof}
	
	The next result can be thought of as the dual property of minimal elements.
	
	\begin{theorem}
		Let $x\in Q$. There exists a maximal strongly hollow element less than $x$.
	\end{theorem}
	
	\begin{proof}
		Let \(S:=\{k\in Q\mid  k\leq x \:\: \text{and} \:\: k\; \text{strongly hollow}\}\). Obviously \(S\neq\emptyset\) as $0\in S$. Let $\tilde{S}\subset S$ be a chain. We claim that $\bigvee\tilde{S}\in S$. It is clear that \(\bigvee\tilde{S}\leq x\). Now, suppose \(\bigvee\tilde{S}\leq r\vee s\), where \(r,s\in Q\). Suppose  $\bigvee\tilde{S}\nleq r$ and \(\bigvee\tilde{S}\nleq s\). Notice that the join $\bigvee F$ of every finite subset $F\subseteq\tilde{S}$ is strongly hollow. In particular, let $k\in\tilde{S}$. Since $k\leq\bigvee\tilde{S}\leq r\vee s$, we have that $k\leq r$ or $k\leq s$. Now consider $A:=\{k\in\tilde{S}\mid k\nleq r\}\subset\{k\in\tilde{S}\mid k\leq s\}$ and \(B:=\{k\in\tilde{S}\mid k\nleq s\}\subset\{k\in\tilde{S}\mid k\leq r\}\), both of which are non-empty by assumption and have empty intersection. Let $k_1\in A$ and $k_2\in B$. We have that $k_1\vee k_2\leq r\vee s$, and by Proposition \ref{T2.3}, we have that either $k_1\vee k_2\leq r$ or $k_1\vee k_2\leq s$, both of which are contradictions. Thus $\bigvee\tilde{S}$ is strongly hollow, and as a result, is an upper bound for $\tilde{S}$ in $S$. By Zorn's lemma, then $S$ must have a maximal element.
	\end{proof}
	\smallskip
	
	\section{Hollowness in special lattices}
	\label{hsl}
	
	In this section, we will study how strongly hollow and completely strongly hollow elements behave in multiplicative lattices with additional conditions. We will begin by giving a characterization of quasi-local $r$-lattices. Then we examine how strongly hollow elements characterize reduced and semi-simple lattices as well as both Gelfand and Pr\"ufer lattices.

	We start by giving a lattice version of \cite[Theorem~2.6]{Heinzer-Ratliff-Rush-2002}.
	\begin{proposition}\label{1}
		Let $(Q,m)$ be a quasi-local weak $r$-lattice. Then suppose $i\in Q$ is strongly irreducible and $i<(i:m)$. Then:
		\begin{enumerate}
			\item $(i:m)$ is principal,
			
			\item $i=(i:m)$,
			
			\item for each $j\in Q$, either $j\leq i$ or $(i:m)\leq j$.
		\end{enumerate}
	\end{proposition}

	\begin{proof}
		(1) Since $i<(i:m)$, there exists a principal element $x\in Q$ such that $x\leq (i:m)$ and $x\nleq i$. Suppose $x\neq (i:m)$. Then, there exists a principal element $y\in Q$ with $y\leq (i:m)$ and $y\nleq x$. Consider \(y\wedge x= (x:y)\cdot y\). $(Q,m)$ is quasi-local, and so $(x:y)\leq m$. Since $y\leq (i:m)$, we have that $(x:y)y\leq i$. Therefore, \(x\wedge y \leq i\), and so $y\leq i$. Therefore, \[\{j\in Q \mid j\leq (i:m)\}=\{j\in Q\mid  j\leq i\}\cup\{j\in Q\mid j\leq x\}.\] Suppose, by way of causing a contradiction, that $\{j\in Q \mid j\leq (i:m)\}\neq \{j\in Q\mid j\leq i\}$ and $\{j\in Q \mid j\leq (i:m)\}\neq \{j\in Q\mid j\leq x\}$. Then, there exist $c,c'\in Q$ with $c\leq i$,$c\nleq x$ and $c'\leq x$,$c'\nleq i$. Thus $c\vee c' \leq (i:m)$ while $c\vee c'\nleq x$ and $c\vee c'\nleq i$, which is a contradiction. Therefore, either $\{j\in Q \mid j\leq (i:m)\}=\{j\in Q\mid j\leq x\}$ or $\{j\in Q \mid j\leq (i:m)\}=\{j\in Q \mid j\leq i\}$, equivalently, $(i:m)=i$ or $(i:m)=x$. We must have $(i:m)=x$ (since $i\ngeq (i:m)$) and is principal. 
		
		(2) We have that $x\cdot m=(i:m)\cdot m\leq i$, and so $(i:x)=m$. Since $i<x$ we obtain $i=x\cdot(i:x)=x\cdot m=(i:m)\cdot m$. 
		
		(3) Let $j\in Q$ and suppose $j\nleq i$. If $x\nleq z$, then there exist a principal element $z\in Q$ with $z\leq j$ and $z\nleq x$. Then $x\wedge z=x(z:x)\leq x\cdot m=i$ giving $z\leq i$, which is a contradiction. Therefore, $z\geqslant x$.
	\end{proof}

	\begin{lemma}\label{2}
		Let $(Q,m)$ be a quasi-local lattice and $0\neq a\in Q$ a completely strongly hollow and principal element. Then $\kappa(a)<(\kappa(a):m)$.
	\end{lemma}

	\begin{proof}
		With the intent of obtaining a contradiction, suppose \(\kappa(a) = (\kappa(a):m)\). Then $x\cdot m\leq \kappa(a) \Rightarrow x\leq \kappa(a) \Rightarrow x\ngeq a$. Therefore, $a\cdot m\nleq \kappa(a)$, and so $a\cdot m\geq a$ which gives $a=a\cdot m$. Then $1=m\vee(0:a)$ and so $(0:a)=1$ and $a=0$ which is a contradiction. Therefore $\kappa(a)<(\kappa(a):m)$.
	\end{proof}
	
	\begin{proposition}\label{3}
		Let $(Q,m)$ be a quasi-local principally generated C-lattice. If $a\in Q$ is a non-zero completely strongly hollow element, then $a$ is comparable to every element in Q. In this case $a=\kappa(a):m$.
	\end{proposition}

	\begin{proof}
		By  Lemma \ref{2}, $\kappa(a)$ satisfies the requirements for Proposition \ref{1}. Let $j\in Q$. Then either $j\leq \kappa(a)$ or $j\geq (\kappa(a):m)$. Notice that, if $j\nleq \kappa(a)$ then $j\geq a$. Therefore, we have that $a=(\kappa(a):m)$. Therefore, if $j\ngeq a$, we have that $j\leq \kappa(a)<(\kappa(a):m)=a$.
	\end{proof}

	We are now in a position to characterize quasi-local weak $r$-lattices in terms of completely strongly hollow elements.
	
	\begin{theorem}\label{4}
		Let $Q$ be a weak $r$-lattice. Then the following are equivalent:
		\begin{enumerate}
			\item $Q$ is a quasi-local lattice, 
			\item Let $a$ be a principal element of $Q$. Then $a$ is completely strongly hollow if and only if $a$ is comparable to every element in $Q$.
		\end{enumerate}
	\end{theorem}
	
	\begin{proof}
		$(1)\Rightarrow (2)$: Let $a\in Q$ be a principal element. Suppose $a$ is completely strongly hollow. Then by Proposition \ref{3}, $a$ is comparable to every element in $Q$. Conversely, suppose $a$ is comparable to every element in $Q$ and $a\leq r\vee s$ for some $r,s\in Q$. If $a=0$, then we are done, and so, suppose $a\neq 0$. Then if $a\nleq r$ and $a\nleq s$, we have that $a\geqslant r$ and $a\geqslant s$. Therefore, there exist $c,c'\in Q$ such that $ac=r$ and $ac'=s$. Since $a=a(c\vee c')$, we have $1=(c\vee c')\vee (0:a)$. But $(0:a)\neq 1$ since $a\neq 0$. Therefore, either $c=1$ or $c'=1$, and we are done.  
		
		$(2)\Rightarrow (1)$: Notice that $1$ is a principal element which is comparable to every other element in $Q$. Therefore, we have that $1$ is completely strongly hollow. Hence, by Proposition \ref{P2.11}, $Q$ is a quasi-local lattice.
	\end{proof}

	The following proposition shows sufficient conditions on a multiplicative lattice that make minimality equivalent to hollowness. 
	\begin{proposition}\label{P3.1}
		Let Q be a semi-simple multiplicative lattice. Then a non-zero element $a\in Q$ is strongly hollow if and only if $a$ is minimal.
	\end{proposition}
	
	\begin{proof}
		Let $a\in Q$ be strongly hollow and $0<x\leq a$. Since $Q$ is semi-simple, there exists a maximal element $m\in Q$ such that $m\ngeq x$. Since $a\leq m\vee x$, we have that $a\leq x$ (because $x\leq a \Rightarrow a \nleq m$), and thus $a=x$. Conversely, assume $a$ is minimal and take $a\leq r \vee s$. Let $m\in\mathcal{M}$. Then $a^2\leq m$ implies $a\leq0$, and therefore, $a\neq0 \Leftrightarrow a^2\neq0$. Hence, $a$ must be idempotent. As a result, $a=a(r\vee s)$. If $ar\leq a$,  then $ar=a$, in which case, we are done; or if $ar=0$, then we must have $a=as$.  The same result holds with the other assumption, and thus $a$ is strongly hollow.
	\end{proof}
	
	The same proof of the `only if' implication above is given in \cite[Lemma~5.6]{Abs-CIT}. 
	Next, we study strongly hollow elements in Gelfand $C$-lattices. But first, a few lemmas.
	
	\begin{lemma}\label{L3.2}
		Let $Q$ be a semi-simple $C$-lattice. Then a non-zero element $e$ is minimal and complemented if there exists a unique element $m\in \mathcal{M}$ such that $e\nleq m$.
	\end{lemma}
	
	\begin{proof}
		Suppose $m\in \mathcal{M}$ and $e\leq n$ for all $n\in \mathcal{M}\setminus\{m\}$ while $e\nleq m$. Then $e\wedge m\leq J(Q)=0$ and $e\vee m=1$, and hence $e$ is complemented. Let \(0<x\leq e\). Then $x\wedge m\leq e\wedge m=0$ implies $x\nleq m$. Therefore, in $Q_m$, we have that $e_m=1_m=x_m$. Consider $Q_n$ for $n\in \mathcal{M}\setminus\{m\}$. Since $Q$ is a $C$-lattice we have that \[m=\bigvee\{c\mid \text{c is compact and }c\leq m\}.\] Now since $m\nleq n$, there must exists a compact element $c\leq m$ such that $c\nleq n$. Since $a\cdot m\leq a\wedge m=0$, we have that $a\cdot c=0=x\cdot c$. Thus, in $Q_n$, we have  $0_n=(a\cdot c)_n=a_n\cdot c_n=a_n\cdot1_n$ and the same holds for $x$. Hence $a_n=x_n$, for all maximal element $n\in \mathcal{M}$, and therefore, $x=a$.
	\end{proof}
	
	\begin{lemma}\label{L3.3}
		Any maximal strongly hollow element in a non-local lattice is idempotent.
	\end{lemma}
	
	\begin{proof}
		Let $m,n\in \mathcal{M}$ be two distinct maximal elements. Thus $m^2\nleq n$, and hence $m\leq m^2\vee n$, which gives $m= m^2$.
	\end{proof}
	Recall that a Gelfand ring (often called a $pm$-ring) is a ring in which every prime ideal is contained in a unique maximal ideal; we will call a lattice Gelfand if the same statement holds for its prime elements. Before we proceed, we need to prove a lattice version of a part of \cite[Theorem 1.2]{DeMarco-Orsatti-1971}.
	
	\begin{lemma}\label{T2}
		If $Q$ is a Gelfand $C$-lattice, then $\mathcal{M}$ is a $T_2$-space.    
	\end{lemma}
	\begin{proof}
		Let $m_1,m_2\in \mathcal{M}$ such that $m_1\neq m_2$. Consider the set $A=\{st\mid s\nleq m_1 \text{ and } t\nleq m_2\}$. If $0\notin A$, then there exists $p\in \operatorname{Spec}(Q)$ such that $\twoheaddownarrow p\cap A=\emptyset$ \cite[Lemma~1.1]{Georgescu-Voiculescu-1989}. Then every compact element $c\leq p$ is such that $c\leq m_1$ and $c\leq m_2$. But since $Q$ is a $C$-lattice, this implies $p\leq m_1$ and $p\leq m_2$, which is a contradiction. Therefore, $0\in A$ which implies  there exists $s\nleq m_1$ and $t\nleq m_2$ such that $st=0$, and hence $t\leq m_1$ and $s\leq m_2$. Therefore, the open sets $D(s)=\{m\in \mathcal{M}\mid s\nleq m\}$ and $D(r)=\{m\in \mathcal{M}\mid r\nleq m\}$ are disjoint open sets that separate $m_1$ and $m_2$.
	\end{proof}
	
	\begin{theorem}\label{T3.4}
		Let $Q$ be a Gelfand $C$-lattice which is semi-simple. Let $a$ be a non-zero element. Then the following are equivalent:
		
		\begin{enumerate}
			\item $a$ is strongly hollow,
			\item $x \cdot y \neq 0$ for all $0<x,y \leq a$,
			\item $a$ is minimal
		\end{enumerate}
	\end{theorem}
	
	\begin{proof}
		We already have $(1)\Leftrightarrow(3)$, and $(3)\Rightarrow(2)$ is trivial. Therefore, it is sufficient to show $(2)\Rightarrow(3)$. Let $a\in Q$ be a non-zero compact element such that $x\cdot y \neq 0$, for all non-zero elements $x,y\leq a$. To show that $a$ is minimal, we will show that $a$ is less than every maximal element except for a unique maximal element. Notice that since $Q$ is semi-simple, we have that $a\nleq J(Q)$, and hence, there must be at least one maximal element $m$ such that $a\nleq m$. With the intent of obtaining a contradiction, suppose there exist distinct $m_1$, $m_2\in \mathcal{M}$ such that $a\nleq m_1$ and $a\nleq m_2$. Then by Lemma \ref{T2}, there exist disjoint open sets $G,H\subset\mathcal{M}$ with $m_1\in G$ and $m_2\in H$. Then,  $\mathcal{M}\setminus G$ and $\mathcal{M}\setminus H$ are closed sets, and thus, are the intersection of basic closed sets; $\mathcal{M}\setminus G =\bigcap V(x_\alpha) $. Therefore, there exist  non-zero $b_1$ and $b_2$ such that $\mathcal{M}\setminus G \subseteq V(b_1)$ and $\mathcal{M}\setminus H\subseteq V(b_2)$ with $b_1\nleq m_1$ and $b_2\nleq m_2$. Notice that $a\cdot b_i\nleq m_i$ since $a\nleq m_i$ and $b_i\leq m_i$ implies $b_i\leq J(Q)=0$. Hence, $a\cdot b_i\neq 0$, $a\cdot b_i\leq a$. Thus, $0=(a\cdot b_1)(a\cdot b_2)\leq J(Q)$, contradicts the assumption that the product of any two non-zero elements less than $a$ is non-zero. Therefore, there cannot exist two maximal elements which are not greater than $a$, hence by Lemma \ref{L3.2}, $a$ is minimal.
	\end{proof}
	
	\begin{corollary}\label{C3.6}
		Let $Q$ be a semi-simple $C$-lattice which is weak meet principally generated. Then for $a\neq 0$, the above properties are equivalent to $a$ being a uniform element.
	\end{corollary}
	
	\begin{proof}
		It is clear that property (2) from Theorem \ref{T3.4} implies that $a$ is uniform. We will now show that if $a$ is a uniform element, then $a$ satisfies (2). Let $0<x,y\leq a$ be two weak meet principal elements. Then $x\wedge y\neq 0$, and thus, there exist $c,d \in Q$ such that $xc=yd=x\wedge y\neq 0$. Since $Q$ is semisimple,  there are no nilpotent elements and as such \[0\neq(x\wedge y)^2=(cx)(dy)=(cd)(xy),\] which gives $xy\neq0$. Now, let $0<x,y\leq a$. Since $x$ and $y$ are the joins of the weak meet principal elements below them, $xy=0$ implies that the product of all such pairs of weak meet principal elements below them must be $0$, which is a contradiction, and hence, $xy\neq 0$. 
	\end{proof}
	
	Recall that an \emph{$I_0$ lattice} is a lattice in which for every element $x\neq J(Q)$, there exists a complemented element $0<c\leq x$. A non-zero idempotent element $e$ is called \emph{primitive} if it cannot be written as the join of two non-zero orthogonal idempotent elements.
	
	\begin{proposition}
		Let $Q$ be an $I_0$ lattice with $1$ compact and $e$ a non-zero compact, complemented, neutral element of $Q$. Then $e$ is completely strongly hollow if and only if $e$ is primitive.
	\end{proposition}
	
	\begin{proof}
		Suppose $e$ is strongly hollow. Then, by Corollary \ref{T2.18},  there exists a unique maximal element $0<x\leq e$. Let $a\vee b =e$ with $a,b$ two orthogonal idempotent elements of $e$. Thus, by locality, $a=e$ or $b=e$. Without loss of generality, assume that $a=e$. Then $b\cdot e=0$, but then $ab\vee b^2 =0$ gives that $b=0$, which is a contradiction. Therefore, $e$ is primitive. 
		
		For the converse, assume that $e$ is primitive. Let $x<e$ and suppose $x\nleq J(Q).$ Then there exists a complemented element $0<c\leq x$ with complement $d$. Notice that $c^2=c$ and $d^2=d$ with $c\cdot d=0$. Hence \((ec)^2=ec\), \((ed)^2=ed\) and \((ec)(ed)=0\). Since $c\vee d=1$, we have  \((e\cdot c)\vee(e\cdot d)=e\). Since $e$ is primitive, we have  $ec=0$ or $ed=0$. But $e\cdot c \geqslant c^2=c>0$, therefore $ed=0$, which gives $e=ec\leq c\leq x$, is a contradiction. Therefore, $x\leq J(Q)$, and thus, by Theorem \ref{T2.20}, $e$ is strongly hollow.
	\end{proof}
	
	We are now in a position to give a necessary and sufficient condition for the existence of a completely strongly hollow element. Recall that a \emph{Pr\"ufer lattice} is a lattice in which every compact element is principal. It is well known that a lattice is a Pr\"ufer lattice if and only if the localization $L_p$ is totally ordered for every prime $p$ \cite[p. 2]{Pri-Pru}. Notice that every Pr\"ufer lattice which is generated by compact elements (e.g., a $C$-lattice) is principally generated. 
	
	\begin{theorem}\label{T3.8}
		Let $Q$ be a reduced Pr\"ufer $C$-lattice. Then $Q$ has a non-zero completely strongly hollow element if and only if there exists a minimal prime $p$ such that \[\bigwedge_{p\neq q\in \mathcal{P}_0}q\nleq p,\]
		and $Q/p$ is a quasi-local domain.
	\end{theorem}
	
	\begin{proof}
		Let $a$ be a completely strongly hollow element of $Q$. Then $a$ is compact and principal, and hence, by Proposition \ref{T2.14}, we have $Q/(0:a)$ is a quasi-local lattice. Since $Q$ is reduced, we have that \begin{align*}
			0=\sqrt{0}=\bigwedge_{q\in \mathcal{P}_0}q&=\left(\bigwedge\{p\in \mathcal{P}_0| \:p\ngeq a\}\right)\wedge\left(\bigwedge\{p\in \mathcal{P}_0| \:p\geq a\}\right) \\
			&=\left(\bigwedge\{p\in \mathcal{P}_0| \:p\ngeq a\}\right)\wedge\sqrt{a} \\
			&\geqslant \left(\bigwedge\{p\in \mathcal{P}_0| \:p\ngeq a\}\right)\wedge a.
		\end{align*}
		Therefore, \(0=a\cdot \bigwedge_{a\nleq q\in \mathcal{P}_0}q\). Thus $\bigwedge_{a\nleq q\in \mathcal{P}_0}q\leq (0:a)$. Since $a\cdot (0:a)=0$, we have  $(0:a)\leq q$, for all $q\in \mathcal{P}_0$ with $q\ngeq a$, which gives $(0:a)\leq\bigwedge_{a\nleq q\in \mathcal{P}_0}q$. Thus $\bigwedge_{a\nleq q\in \mathcal{P}_0}q=(0:a)$. Now, let $m$ be a maximal element in $Q$, and with the intent of obtaining a contradiction, suppose there exist $q_1,q_2\in \mathcal{P}_0$ such that $q_1\leq m$ and $q_2\leq m$. Considering the localization $Q_m$, we find \[(q_i)_m=\bigvee\{x\in Q\mid x\cdot s\leq q_i\:,\: s\nleq m\text{ and } s\text{ compact}\}.\]
		But, if $x\cdot s\leq q_i$, then $x\leq q_i$ since $s\nleq m$. Thus \((q_i)_m=q_i\). Notice that since we have $Q_m$ totally ordered, either $(q_1)_m\leq (q_2)_m$ or vice-versa. Then $q_1$ and $q_2$ are comparable, which contradicts minimality, and hence,  no two minimal primes are less than the same maximal element. If there are two minimal primes $q_1,q_2\ngeq a$, there must exist two maximal elements $n$ and $m$, both of which are greater than $(0:a)$, which contradicts the locality of $Q/(0:a)$. Therefore, there is only one minimal prime $q\ngeq a$, and thus $(0:a)=q$. Therefore, $Q/(0:a)$ is a quasi-local domain and $0<a\leq \bigwedge_{p\neq q\in \mathcal{P}_0}q\nleq p$.
		
		For the converse, suppose there exists a minimal prime $p$ such that $Q/p$ is a quasi-local domain and $\bigwedge_{p\neq q\in \mathcal{P}_0}q\nleq p$. Since $Q$ is compactly generated, select a non-zero compact element $a\leq\bigwedge_{p\neq q\in \mathcal{P}_0}q$. Then such an element is principal and $a\nleq p$ (since then $a\leq\bigwedge_{q\in \mathcal{P}_0}q=0$). Therefore, we have  $a\wedge p=0$ and therefore $p\leq (0:a)$; but then, by the same argument as before, we have  $p=(0:a)$. Hence, $Q/(0:a)$ is a quasi-local domain. Let $m$ be a maximal element such that $(0:a)\nleq m$. Since $Q$ is compactly generated,  there must exist a compact element $c\leq(0:a)$ with $c\nleq m$ and $c\cdot a=0$. Consider \(a_{m}\) in $Q_{m}$. Then \[ 0_m=(a\cdot c)_m=a_m\cdot c_m=a_m\cdot 1_m.\] Therefore, $a_m=0_m$. Let $m$ be the unique maximal element greater than $(0:a)$ and let $a_m\leq s_m\vee r_m$. Suppose $a_m\nleq r_m$ and $a_m \nleq s_m$. Since $Q_m$ is totally ordered, we have  $r_m\leq a_m$ and $s_m\leq a_m$. Moreover, $a_m=r_m\vee s_m$. Since $a$ is principal in $Q$, and thus principal in $Q_m$, there must exist $c_m, b_m\in Q_m$ such that $a_mc_m=r_m$ and $a_mb_m=s_m$. Therefore, $a_m=a_m(c_m\vee b_m)$ and hence, $1_m=(0_m:a_m)\vee(c_m\vee b_m)$. Since $Q_m$ is quasi-local, this gives t $(0:a)_m=1_m$ or $c_m\vee b_m=1_m$. But $(0:a)_m=1_m$ implies that $a_m=0_m$, and hence, $a=0$, which is a contradiction. Therefore, we have that $b_m\vee c_m=1_m$, and thus, $b_m=1_m$ or $c_m=1_m$; but then $a_m=r_m$ or $a_m=s_m$, which is a contradiction. Therefore, we must have  $a_m\leq r_m$ or $a_m\leq s_m$. Then by Proposition \ref{T2.15}, we obtain that $a$ is strongly hollow in $Q$.
	\end{proof}
	\smallskip
	
	\section{Decomposition of multiplicative lattices}
	\label{dml}
	
	In this section, we outline how we can decompose both elements and multiplicative lattices through products involving completely strongly hollow and strongly hollow elements. In order to proceed, we will need a few more definitions. 
	We show how completely strongly hollow and strongly hollow elements occur in the direct product of lattices and give requirements for lattices to have a finite number of maximal and minimal elements based on the existence of hollow elements. 
	
	\begin{proposition}\label{T4.2}
		Let $Q=Q_1\times\cdots\times Q_n$ be a direct product of multiplicative lattices and $a\in Q$. Then a is strongly hollow if and only if there exists $i\in \{1,...,n\}$ with $a=\langle 0,...,a_i,...,0\rangle$ and $a_i$ in the i-th index such that $a_i\in Q_i$ is strongly hollow. In this case, we have  $\kappa(a)=\langle 0,...,\kappa_{Q_i}(a_i),...,0\rangle $ and \((\kappa(a):a)=\langle 1,...,(\kappa_{Q_i}(a_i):a_i),...,1\rangle\).
	\end{proposition}
	\begin{proof}
		Let $a$ be strongly hollow in $Q$. Then $a=\langle a_1,..,a_n\rangle $. Let $\bar{a_i}$ denote the element with $a_i$ in the $i$-th index and $0$ everywhere else. Then, we have 
		$a=\bigvee_{i=1}^n\bar{a_i}$. Since $a$ is strongly hollow, we must have $a\leq \bar{a_i}$, for some $i\in\{1,...,n\}$ and therefore, $a=\langle 0,...,a_i,...\rangle $. Now, let $x,y\in Q_i$ and consider $\bar{x},\bar{y}\in Q$. We have  $a_i\leq x\vee y$ implies $a\leq\bar{x}\vee\bar{y}$, and hence,  $a_i\leq x$ or $a_i\leq y$ (since $a$ is strongly hollow). Therefore, $a_i$ is strongly hollow in $Q_i$.
		
		Conversely, suppose $a_i$ is strongly hollow in $Q_i$ and consider $\bar{a_i}$. Suppose $\bar{a_i}\leq x\vee y$, for $x,y\in Q$. Therefore, $x$ and $y$ are zero in every index but the $i$-th index. Since $a_i$ is strongly hollow, we have  $a\leq x_i$ or $a_i\leq y_i$, and hence, $a\leq x$ or $a\leq y$. Notice that, if $x\ngeq a$, then $x_i\ngeq a_i$, and for all $j\neq i$, there are no restrictions on $x_j$. Therefore, we have  $\kappa(a)=\langle 1,...,\kappa_{Q_i}(a_i),...,1\rangle$, and as result, we also have  \[(\kappa(a):a)=\langle 1,...,(\kappa_{Q_i}(a_i):a_i),...,1\rangle.\qedhere\]
	\end{proof}
	
	The same proof follows for completely strongly hollow elements by just considering domination of $a$ by arbitrary joins. We now describe the representation of elements of a lattice in terms of completely strongly hollow elements.
	
	\begin{definition}
		An element $x\in Q$ is called \emph{semi-pseudo strongly hollow representable} whenever $x$ can be represented as the join of completely strongly hollow elements of $Q$ and {completely strongly hollow representable} whenever $x$ can be represented as the join of a finite number of completely strongly hollow elements of $Q$. A representation $x=\bigvee_{i\in \Omega} a_i$ is called \emph{minimal}, whenever $a_j\nleq \bigvee_{j\neq i\in \Omega}a_i$ for each $j\in \Omega$.
	\end{definition}
	
	Any completely strongly hollow element is completely strongly hollow representable, and any completely strongly hollow representable element is semi-pseudo strongly hollow representable.
	
	\begin{proposition}
		Let $Q=Q_1\times\cdots \times Q_n$. Then every element of Q is semi-pseudo strongly hollow representable if and only if for every $i\in\{1,...,n\}$, every element of $Q_i$ is semi-pseudo strongly hollow representable.
	\end{proposition}
	
	\begin{proof}
		We will only show it is true for $n=2$. Let $Q=R\times S$, and suppose that every element of $Q$ is semi-pseudo strongly hollow representable. Let $i\in R$ be non-zero. Then $\langle i,0\rangle=\bigvee_{j\in \Omega} \langle x_j,0\rangle $ with $\langle x_j,0\rangle$ completely strongly hollow, for every $j\in \Omega$; and as a result, $x_j$ must be completely strongly hollow in $R$. Therefore, $i=\bigvee_{j\in \Omega}x_j$ is a semi-pseudo strongly hollow representation of $i$ in $R$. The same result follows for $S$. 
		
		For the converse, assume that every element $x\in R$ and every element $y\in S$ is semi-pseudo strongly hollow representable. Let $i=\langle x,y\rangle\in Q$ and $x=\bigvee_{l\in L} x_l$ and $y=\bigvee_{j\in J} y_j$ semi-pseudo strongly hollow representations of $x$ and $y$. Then, we have  $i=\langle \bigvee_{l\in L} x_l ,0\rangle \vee\langle 0,\bigvee_{j\in J} y_j\rangle$, and by Proposition \ref{T4.2}, each term in both the left join and right join is completely strongly hollow in $Q$. Therefore, $i=\langle\bigvee_{l\in L} x_l ,0\rangle \vee\langle 0,\bigvee_{j\in J} y_j\rangle $ is a semi-pseudo strongly hollow representation of $i$.
	\end{proof}
	
	By the same argument, the above result holds when substituting completely strongly hollow representable for semi-pseudo strongly hollow representable. We now show a form of uniqueness for a representation of an element by completely strongly hollow elements.
	
	\begin{theorem}
		Let $\bigvee_{\omega\in \Omega}a_\omega$ and $\bigvee_{\lambda\in \Lambda}b_\lambda$ be two minimal semi-pseudo strongly hollow representations (respectively, completely strongly hollow) for an  $i\in Q$. Then for each $\omega\in\Omega$, there exists a unique $\lambda\in\Lambda$ such that $a_\omega=b_\lambda$ and vice versa.
	\end{theorem}
	
	\begin{proof}
		Let $\omega\in\Omega$. Since $a_\omega\leq\bigvee_{\lambda\in\Lambda}b_{\lambda}$, there must exists $\lambda\in\Lambda$ such that $a_{\omega}\leq b_\lambda$; but the same  holds for $\lambda$, giving $a_{\omega '}\geqslant b_\lambda$, and hence, $a_{\omega '}\geqslant a_\omega$. Since the representation is minimal, we must have $\omega=\omega'$ and therefore, $a_\omega=b_\lambda$. The same argument by minimality as above gives the uniqueness.
	\end{proof}
	
	\begin{remark}
		We obtain from the uniqueness of the representation that every completely strongly hollow representable element has a unique minimal completely strongly hollow representation.
	\end{remark}
	
	\begin{lemma}\label{L4.5}
		Let $Q$ be a multiplicative lattice with $1$ compact and semi-pseudo strongly hollow representable. Then $Q$ is semi-local.
	\end{lemma}
	
	\begin{proof}
		Since $1$ is compact and semi-pseudo hollow representable, we find $1$ is completely strongly hollow representable. Let $1=\bigvee_{i=1}^na_i$ be a completely strongly hollow representation. Then, for each $i\in\{1,...,n\}$, we have  either $a_i\leq J(Q)$ or there is a unique maximal element $m\ngeq a_i$. If $a_i\nleq J(Q)$, for all $i\in\{1,...,n\}$, then at most every $a_i$ has a unique $m_i$ such that $a_i\nleq m_i$. Then, if there were an additional $m_{m+1}$, we would have  all $a_i\leq m_{n+1}$ giving that $m_{n+1}=1$. Therefore, there are at most $n$ maximal elements in $Q$.
	\end{proof}
	
	We denote by $\mathrm{Soc}(Q)$ the \emph{socle} of $Q$, which is the join of all minimal elements of $Q$.
	
	\begin{lemma}\label{L4.6}
		Let $Q$ be a multiplicative lattice such that every non-zero element is completely strongly hollow representable by non-zero completely strongly hollow elements. Then every minimal element of Q is completely strongly hollow, and the set of all minimal elements of Q is finite.
	\end{lemma}
	
	\begin{proof}
		Let $i$ be a minimal element of $Q$ and $i=\bigvee_{j=1}^na_j$ a completely strongly hollow representation of $i$. Since $i$ is minimal and the representatives are non-zero, we have  $a_j=i$, and so $i$ is completely strongly hollow. If $\mathrm{Soc}(Q)=0$, then we are done. Otherwise, consider a minimal representation of $\mathrm{Soc}(Q)=\bigvee_{j=1}^na_j$. Notice that, since each $i$ is completely strongly hollow,  they are in one-to-one correspondence with the minimal representation by uniqueness, and therefore, there is a finite number of minimal elements in $Q$.
	\end{proof}
	\smallskip
	
	\section{Completely Strongly Hollow elements in Weak Noether lattices}
	\label{cwn}
	
	In this section, we provide sufficient conditions for a lattice to be a weak Noether lattice.
	A weak $r$-lattice which satisfies the ascending chain condition is known as a \emph{weak Noether lattice}. A \emph{principal element lattice} is a multiplicative lattice in which every element is a principal element. For basic information and facts about weak $r$-lattices and weak Noether lattices, we refer the reader to \cite{AJ-PE-1996}. Notice that in a weak $r$-lattice every completely strongly hollow element is principal and compact. 
	
	We first introduce some lemmas that will be used often in the sequel.
	
	\begin{lemma}\label{L5.1}
		Let $Q$ be a weak $r$-lattice. If $(Q,m)$ is a quasi-local principal element lattice, then $Q$ is a weak Noether chain lattice. 
	\end{lemma}
	
	\begin{proof}
		By \cite[Theorem~1.1]{AJ-PE-1996}, we have that $Q$ is a weak Noether lattice. Since every element is principal,  every compact element is principal, and thus $Q$ is a pr\"ufer lattice. Therefore, $Q_m$ is totally ordered. Since $Q$ is a quasi-local lattice with unique maximal element $m$, we have  $Q_m=Q$, and therefore, $Q$ is a chain lattice. 
	\end{proof}
	
	\begin{lemma}\label{L5.2}
		Let $Q$ be a weak Noether chain lattice. Then every element of $Q$ is completely strongly hollow.
	\end{lemma}
	
	\begin{proof}
		Let $a\in Q$ and consider $a\leq\bigvee_{i\in\Omega}r_i$. Since $Q$ is a Noether lattice, it satisfies the ascending chain condition, which implies that every chain has a maximal element and thus every maximal element is the maximum element (since $Q$ is a chain). Therefore $\bigvee_{i\in\Omega}r_i=r_J$, for some $j\in\Omega$. Hence $a$ is completely strongly hollow.
	\end{proof}
	
	\begin{proposition}\label{P5.3}
		Let $Q$ be a weak $r$-lattice. Then every element of $Q$ is a finite product of completely strongly hollow elements if and only if $Q$ is a weak Noether chain lattice.
	\end{proposition}
	
	\begin{proof}
		The backwards implication follows from Lemma \ref{L5.2}. Conversely, since every element is a finite product of completely strongly hollow elements, we have that $1$ is as well, and therefore, since $1$ can only be decomposed into a product with itself, we have that $1$ is completely strongly hollow. Therefore, $Q$ is quasi-local. Since the finite product of principal elements is still principal, we have that $Q$ is a principal element lattice. Therefore, $Q$ is a weak Noether chain lattice by Lemma \ref{L5.1}.
	\end{proof}
	
	\begin{proposition}\label{P5.4}
		Let $Q$ be a weak $r$-lattice. Then every non-zero element of $Q$ is strongly hollow if and only if $Q$ is a chain lattice. In particular, every non-zero element of $Q$ is completely strongly hollow if and only if $Q$ is a weak Noether chain lattice.
	\end{proposition}
	
	\begin{proof}
		If $Q$ is a chain,  then it is clear that every element is strongly hollow. Conversely, suppose every element is strongly hollow and let $a,b\in Q$ such that $a\nleq b$. Since $a\vee b\leq a\vee b$, we have $a\vee b\leq a$ or $a\vee b\leq b$, and hence $b\leq a$ and hence $Q$ is a chain lattice. The complementary result for completely strongly hollow elements follows clearly from Lemma \ref{L5.2} and Proposition \ref{P5.3}.
	\end{proof}
	
	A \emph{$\pi$-lattice} is a lattice which is generated by a set of elements which are all the products of prime elements. A \emph{ZPI-lattice} is a lattice in which every element is a product of prime elements. A \emph{UFD lattice} is a principally generated multiplicative lattice domain in which every principal element is a product of principal primes. A \emph{special principal element lattice} is a principal element lattice with a unique maximal element whose powers generate the entire lattice.
	
	\begin{theorem}\label{T5.5}
		Let $Q$ be a weak $r$-lattice. Then the following are equivalent:
		
		\begin{enumerate}
			\item Every non-zero prime element is completely strongly hollow.
			
			\item $Q$ is either the finite boolean algebra $B_4$ or a weak Noether chain lattice which is either a UFD or a special principal element lattice.
			
			\item Every non-zero element of $Q$ is completely strongly hollow.
			
			\item Every non-zero element of $Q$ is the finite product of completely strongly hollow elements.
		\end{enumerate}
	\end{theorem}
	
	\begin{proof}
		We have  $(4)\Leftrightarrow (3)\Rightarrow(1)$ by Lemma \ref{L5.2} and Proposition \label{P5.3}. Notice that if $Q$ is a finite boolean algebra $B_4$, every element is trivially completely strongly hollow, and if $Q$ is a Noether chain lattice, then by Lemma \ref{L5.2}, every element is completely strongly hollow, and hence, $(2)\Rightarrow(3)$. We will now show that $(1)\Rightarrow(2)$. Suppose every non-zero prime element is completely strongly hollow. Then every prime element is a principal element and thus $Q$ is a principal element lattice (see \cite[p.~586]{Pri-Pru}). Therefore, by \cite[Theorem~1.1]{AJ-PE-1996},  $Q$ is a weak Noether lattice. Now, suppose there exist $3$ distinct maximal elements $m,n,l$ in $Q$. Then, we have  $m\leq n\vee l$, which gives  $m\leq n$ or $m\leq l$, contradicting maximality. Therefore, there are at most $2$ maximal elements in $Q$. Suppose there exist exactly $2$ distinct maximal elements $m,n\in Q$. Consider $x\in Q$ such that $x< m$ and $x\nleq n$. Then $m\leq x\vee n$, giving a contradiction. Therefore, by applying the same argument to $n$, we have that for all $x\in Q$, \(x<m \Leftrightarrow x<n\). Notice that $m^2\nleq n$ gives $m^2=m$, and the same for $n$. Therefore, by \cite[Theorem~2.3]{Abs-CIT}, $Q$ is a finite Boolean algebra. Let $x\notin \{m,n,1\}$, and hence $x<m$ and $x<n$. Since $Q$ is a Boolean algebra, there exists an element $c\in Q$ such that $c\vee x=1$ and $c\wedge x=0$. If $c\neq 1$, then $c$ is less than $m$ or $n$ and $c\vee x\neq 1$. Therefore, $c=1$, and as a result, $x=0$. Therefore, $Q=\{0,1,m,n\}=B_4$. Now suppose $Q$ has a unique maximal element $m$ such that $m^2\neq m$. Then by Lemma $\ref{L5.1}$, we find $Q$ is a weak Noether chain lattice. Since $Q$ is a principal element lattice, we have that $Q$ is a ZPI-lattice, and hence, a $\pi$-lattice (see\cite[Theorem~8]{Pri-Pru}). Therefore, $Q$ is either a special principal element lattice or a domain (see \cite[Corollary~2.1]{AJ-PE-1996}). If $Q$ is a domain, then $Q$ is a UFD (see \cite[Corollary~2.3]{AJ-PE-1996}).
	\end{proof}
	
	\begin{proposition}\label{P5.6}
		Let $Q$ be a weak $r$-lattice. Then the following are equivalent:
		\begin{enumerate}
			\item Every non-zero maximal element of $Q$ is completely strongly hollow.
			
			\item $Q$ is either the finite boolean algebra $B_4$ or $Q$ is a quasi-local lattice whose unique maximal element is principal.
		\end{enumerate}
	\end{proposition}
	
	\begin{proof}
		$(2)\Rightarrow(1)$: Follows from Theorem \ref{T5.5}, Proposition \ref{P2.7}, and the fact that principal elements are compact (see \cite[Theorem~2.1]{And-ACIT-1976}). Now,  suppose every non-zero maximal element of $Q$ is completely strongly hollow. Then, with the intent of obtaining a contradiction, suppose that $Q$ is not local. Then every maximal element of $Q$ is idempotent by Lemma \ref{L3.3}. Therefore, $Q$ is a finite Boolean algebra by \cite[Theorem~2.3]{Abs-CIT}, and by the same argument as in Theorem \ref{T5.5}, $Q=\{0,1,n,m\}=B_4$. Hence, suppose $Q$ is quasi-local. Then the unique maximal element is completely strongly hollow, and thus, principal. 
	\end{proof}
	
\end{document}